\documentclass[11pt]{article}
\usepackage{amssymb, amsthm, amsmath, amscd}
\newtheorem{thm}{Theorem}[section]
\newtheorem{lem}[thm]{Lemma}
\newtheorem{prop}[thm]{Proposition}
\newtheorem{cor}[thm]{Corollary}
\newtheorem{NN}[thm]{}
\theoremstyle{definition}\newtheorem{df}[thm]{Definition}
\theoremstyle{definition}\newtheorem{rem}[thm]{Remark}
\theoremstyle{definition}

\renewcommand{\phi}{\varphi}

\newcommand{\N}{\mathbb{N}}
\newcommand{\Z}{\mathbb{Z}}

\newcommand{\R}{\mathbb{R}}
\newcommand{\C}{\mathbb{C}}
\newcommand{\T}{\mathbb{T}}

\newcommand{\hm}{homomorphism}
\newcommand{\dt}{\delta}
\newcommand{\ep}{\epsilon}
\newcommand{\andeqn}{\,\,\,{\rm and}\,\,\,}
\newcommand{\rforal}{\,\,\,{\rm for\,\,\,all}\,\,\,}
\newcommand{\CA}{$C^*$-algebra}
\newcommand{\SCA}{$C^*$-subalgebra}

\newcommand{\beq}{\begin{eqnarray}}
\newcommand{\eneq}{\end{eqnarray}}
\newcommand{\tforal}{\,\,\,\text{for\,\,\,all}\,\,\,}
\newcommand{\tand}{\,\,\,\text{and}\,\,\,}

\title{Unitaries in a Simple $C^*$-algebra of Tracial Rank One
}
\author{Huaxin Lin
 }
\date{}

\begin{document}

\maketitle

\begin{abstract}
Let  $A$ be a  unital separable simple infinite dimensional \CA\,  with tracial rank no
more than one and with the tracial state space $T(A)$ and let $U(A)$
be the unitary group of $A.$  Suppose that $u\in U_0(A),$ the
connected component of $U(A)$ containing the identity.
 We show that,  for any $\ep>0,$ there exists a selfadjoint
element $h\in A_{s.a}$ such that
$$
\|u-\exp(ih)\|<\ep.
$$
We also study the problem when  $u$ can be approximated by unitaries in $A$ with finite
spectrum.

Denote by $CU(A)$ the closure of the subgroup of unitary group of
$U(A)$ generated by its commutators. It is known that $CU(A)\subset
U_0(A).$ Denote by $\widehat{a}$ the affine function on $T(A)$
defined by $\widehat{a}(\tau)=\tau(a).$  We show that $u$ can be
approximated by unitaries in $A$ with finite spectrum if and only if
$u\in CU(A)$ and $\widehat{u^n+(u^n)^*},i(\widehat{u^n-(u^n)^*})\in
\overline{\rho_A(K_0(A)}$ for all $n\ge 1.$ Examples are given that
there are unitaries in $CU(A)$ which can not be approximated by
unitaries with finite spectrum. Significantly these  results are
obtained in the absence of  amenability.

\end{abstract}

\section{Introduction}
Let $M_n$ be the \CA\, of $n\times n$ matrices and let $u\in M_n$ be a unitary.
Then $u$ can be diagonalized, i.e., $u=\sum_{k=1}^n e^{i\theta_k} p_k,$ where $\theta_k\in \R$
 and $\{p_1, p_2,...,p_n\}$ are mutually orthogonal projections. As a consequence, $u=\exp(i h),$ where
 $h=\sum_{k=1}^n \theta_kp_k$ is a selfadjoint matrix.
Now let $A$ be a unital \CA\, and let $U(A)$ be the unitary group of
$A.$  Denote by $U_0(A)$ the connected component of $U(A)$
containing the identity. Suppose that $u\in U_0(A).$ Even in the case that $A$ has real rank zero,
$sp(u)$ can have infinitely many points and it is impossible to write $u$ as an exponential, in general. However, it was shown
(\cite{LnFU}) that $u$ can be approximated by unitaries in $A$ with
finite spectrum if and only if $A$ has real rank zero. This is an
important and useful feature  for \CA s of real rank zero. In this case, $u$ is a norm limit of exponentials.

Tracial rank for \CA s was introduced (see \cite{Lnplms}) in the
connection with the program of classification of separable
amenable \CA s, or otherwise known as the Elliott program. Unital
separable simple amenable \CA s with tracial rank no more than one
which satisfy the universal coefficient theorem have  been
classified by the Elliott invariant (\cite{EGL} and
\cite{Lnctr1}). A unital separable simple \CA\, $A$ with $TR(A)=1$
has real rank one. Therefore a unitary $u\in U_0(A)$ may not be
approximated by unitaries with finite spectrum. We will show that,
in a unital  infinite dimensional simple \CA\, $A$ with tracial rank no more
than one, if $u$ can be approximated by unitaries in $A$ with
finite spectrum then $u$ must be in $CU(A),$ the closure of the
subgroup generated by commutators of the unitary group. A related
problem is whether every unitary $u\in U_0(A)$ can be approximated
by unitaries which are exponentials. Our first result is to show
that, there are selfadjoint elements $h_n\in A_{s.a}$ such that
$$
u=\lim_{n\to\infty}\exp(ih_n)
$$
(converge in norm).
It should be mentioned that exponential rank has been studied quite extendedly (see \cite{PR}, \cite{Phe},
\cite{Ph2}, \cite{Ph3}, etc.).
In fact, it was shown by N. C. Phillips that a unital simple  \CA\,  $A$
which is an  inductive limit of finite direct sums of \CA s with the form
$C(X_{i,n})\otimes M_{i,n}$ with the dimension of $X_{i,n}$ is bounded
has exponential rank $1+\ep,$ i.e., every unitary $u\in U_0(A)$ can be approximated by unitaries which are exponentials (see \cite{Phe}). These simple \CA s have tracial rank one or zero. Theorem \ref{Ex3} was proved without assuming $A$ is an AH-algebra, in fact, it was proved in the absence of amenability.

Let $T(A)$ be the tracial state space of $A.$ Denote by
$\text{Aff}(T(A))$ the space of all real affine continuous functions
on $T(A).$ Denote by $\rho_A: K_0(A)\to \text{Aff}(T(A))$ the
positive \hm\, induced by $\rho_A([p])(\tau)=\tau(p)$ for all
projections in $M_k(A)$ (with $k=1,2,...$) and for all $\tau\in
T(A).$ It was introduced by de la Harpe and Scandalis (\cite{HS}) a
determinant like map $\Delta$ which maps $U_0(A)$ into
$\text{Aff}(T(A))/\overline{\rho_A(K_0(A))}.$  By a result of K. Thomsen (\cite{Th})
 the de la Harpe and Scandalis
determinant induces an isomorphism between
$\text{Aff}(T(A))/\overline{\rho_A(K_0(A))}$ and $U_0(A)/CU(A).$
We found  out that if $u$ can be approximated by unitaries in $A$
with finite spectrum then $u$ must be in $CU(A).$ But can every
unitary in $CU(A)$ be approximated by unitaries with finite
spectrum? To answer this question, we consider even simpler
question: when can a self-adjoint element in a unital separable
simple \CA\, with $TR(A)=1$ be approximated by self-adjoint
elements with finite spectrum?  Immediately, a necessary condition
for a self-adjoint element $a\in A$ to be approximated by
self-adjoint elements with finite spectrum is that
$\widehat{h^n}\in \overline{\rho_A(K_0(A))}$ (for all $n\in \N$).
Given a unitary $u\in U_0(A),$ there is an affine continuous map
from $\text{Aff}(T(C(\T)))$ into $\text{Aff}(T(A))$ induced by
$u.$ Let $\Gamma(u): \text{Aff}(T(C(\T)))\to
\text{Aff}(T(A))/{\overline{\rho_A(K_0(A))}}$ be the map given by
$u.$ Then it is clear that $\Gamma(u)=0$ is a necessary condition
for $u$ being approximated by unitaries with finite spectrum. Note
that $\Gamma(u)=0$ if and only if $\widehat{u^n+ (u^n)^*}, i\widehat{(u^n-(u^n)^*)}\in
\overline{\rho_A(K_0(A))}$ for all positive integers $n.$ By
applying a uniqueness theorem together with classification results
in simple \CA s, we show that the condition is also sufficient.
From this, we show that a unitary $u\in CU(A)$ can be approximated
by unitaries with finite spectrum if and only if $\Gamma(u)=0.$ We
also show that $\Delta(u)=0$ is not sufficient for $\Gamma(u)=0.$
Therefore, there are unitaries in $CU(A)$ which can not be
approximated by unitaries with finite spectrum (see \ref{No}).
Perhaps more interesting fact is that $\Gamma(u)=0$ does not imply
that $\Delta(u)=0$ for $u\in U_0(A)$ (see \ref{CEX} and \ref{LP}).

\section{Preliminaries}

\begin{NN}
{\rm Denote by ${\cal I}$ the class of \CA s which are finite direct
sums of \SCA s with the form $M_k(C([0,1])$ or $M_k,$ $k=1,2,....$

}
\end{NN}

\begin{df}
{\rm  Recall that a unital simple \CA\, $A$ is said to have tracial
rank no more than  one (or $TR(A)\le 1$), if for any $\ep>0,$ any
$a\in A_+\setminus\{0\}$ and any finite subset ${\cal F}\subset A,$
there exists a projection $p\in A$ and a \SCA\, $B$ with $1_B=p$
such that

{\rm (1)} $\|px-xp\|<\ep$ for all $x\in {\cal F};$

{\rm (2)} ${\rm dist}(pxp, B)<\ep$ for all $x\in {\cal F}$ and

{\rm (3)} $1-p$ is Murry-von Nuemann equivalent to a projection in
$\overline{aAa}.$

Recall that, in the above definition, if $B$ can always be chosen to
have finite dimension, then $A$ has tracial rank zero ($TR(A)=0$).
If $TR(A)\le 1$ but $TR(A)\not=0,$ we write $TR(A)=1.$

 Every unital simple AH-algebra with very slow  dimension growth has tracial
rank no more than one (see Theorem 2.5 of \cite{Lnctr1}). There are
\CA s with tracial rank no more than one which are not amenable.

}
\end{df}

\begin{df}
{\rm Suppose that $u\in U(A).$ We will use ${\bar u}$ for the image
of  $u$ in $U(A)/CU(A).$ If $x, y\in U(A)/CU(A),$ define
$$
{\rm dist}(x, y)=\inf\{\|u-v\|: {\bar u}=x\andeqn {\bar v}=y\}.
$$

Let $C$ be another unital \CA\, and let $\phi: C\to A$ be a unital
\hm. Denote by $\phi^{\ddag}: U(C)/CU(C)\to U(A)/CU(A)$ the \hm\,
induced by $\phi.$
 }
\end{df}

\begin{NN}
{\rm  Let $A$ be a unital separable simple \CA\, with $TR(A)\le 1,$
then $A$ is quasi-diagonal, stable rank one, weakly unperforated
$K_0(A)$ and, if $p, \, q\in A$ are two projections, then $p$ is
equivalent to a projection $p'\le q$ whenever $\tau(p)<\tau(q)$ for
all tracial states $\tau$ in $T(A)$ (see \cite{Lnplms}).

For unitary group of $A,$ we have the following:

{\rm (i)} $CU(A)\subset U_0(A)$ (Lemma 6.9 of \cite{Lnctr1});

{\rm (ii)} $U_0(A)/CU(A)$ is torsion free and divisible (Theorem
6.11 and Lemma 6.6 of \cite{Lnctr1});

%{\rm (iii)} For $u\in CU(A),$ then the exponential length ${\rm
%cel}(u)\le 8\pi$ (Lemma 6.9 of \cite{Lnctr1}).

}
\end{NN}

%\begin{prop}\label{divisible}
%Let $A$  be a unital separable simple \CA\, with $TR(A)\le 1.$ Then
%$\overline{\rho_A(K_0(A))}$ and
%$\text{Aff}(T(A))/\overline{\rho_A(K_0(A))}$ are divisible and
%torsion free.

%\end{prop}

%\begin{proof}
%It is clear that $\overline{\rho_A(K_0(A))}$ is torsion free. That
%$\overline{\rho_A(K_0(A))}$ is divisible follows from the fact that
%$A$ is tracially divisible (see Theorem 5.4 of  \cite{Lnctr1}).

%The fact that $\text{Aff}(T(A))/\overline{\rho_A(K_0(A))}$ is
%divisible and torsion free follows from the fact that
%$\text{Aff}(T(A))$ is divisible and $\overline{\rho_A(K_0(A))}$ is
%divisible and torsion free.

%\end{proof}

\begin{thm}{\rm (Theorem 3.4 of \cite{Lnnhomp})}\label{inject}
Let $A$ be a unital separable simple \CA\, with $TR(A)\le 1$ and let
$e\in A$ be a non-zero projection. Then the map $u\mapsto u+(1-e)$
induces an isomorphism $j$ from $U(eAe)/CU(eAe)$ onto $U(A)/CU(A).$
\end{thm}

%\begin{proof}
%It was shown in Theorem 6.7 of \cite{Lnctr1} that $j$ is a
%surjective \hm. So it remains to show that it is also injective. To
%do this, fix a unitary $u\in eAe$ so that ${\bar u}\in {\rm ker}\,
%j.$ We will show that $u\in CU(eAe).$
%There is an integer $K\ge 1$ such that
%$$
%K[e]\ge [1-e]\,\,\,{\rm in}\,\,\, K_1(A).
%$$
%Let $1>\ep>0.$ Put $v=u+(1-e).$ Then $v\in CU(A).$ In particular,
%$$
%{\rm dist}({\bar v}, \bar{1})<\ep/(2K+2).
%$$
%It follows from Lemma 6.2 of \cite{Lnnhomp} that
%$$
%{\rm dist}({\bar u}, {\bar e})<(2K+1/2)(\ep/(2K+2))<\ep.
%$$
%It follows that
%$$
%u\in CU(eAe).
%$$
%\end{proof}

\begin{cor}\label{c1}
Let $A$ be a unital separable simple \CA\, with $TR(A)\le 1.$ Then
the map $j: a\to {\rm diag}(a, \overbrace{1,1,..,1}^m)$ from $A$ to
$M_n(A)$ induces an isomorphism from $U(A)/CU(A)$ onto
$U(M_n(A))/CU(M_n(A))$ for any integer $n\ge 1.$
\end{cor}

\begin{df}
{\rm Let $u\in U_0(A).$ There is a piece-wise smooth and continuous
path $\{u(t): t\in [0,1]\}\subset A$ such that $u(0)=u$ and
$u(1)=1.$ Define
$$
R(\{u(t)\})(\tau)={1\over{2\pi i}}\int_0^1
\tau({du(t)\over{dt}}u(t)^* )dt.
$$

$R(\{u(t)\})(\tau)$ is real for every $\tau.$  }
\end{df}

\begin{df}\label{Det}
{\rm  Let $A$ be a unital \CA\, with $T(A)\not=\emptyset.$ As in \cite{HS} and \cite{Th},
 define a \hm\, $\Delta:
U_0(A)\to \text{Aff}(T(A))/\overline{\rho_A(K_0(A))}$ by
$$
\Delta(u)=\Pi({1\over{2\pi}}\int_0^1\tau({du(t)\over{dt}}u(t)^*)dt),
$$
where $\Pi: \text{Aff}(T(A))\to
\text{Aff}(T(A))/\overline{\rho_A(K_0(A))}$ is the quotient map and
where $\{u(t): t\in [0,1]\}$ is a piece-wise smooth and continuous
path of unitaries in $A$ with $u(0)=u$ and $u(1)=1_A.$ This is
well-defined and is independent of the choices of the paths. }

\end{df}

The following is a combination of a result of K. Thomsen
(\cite{Th})and the work of \cite{HS}. We state here for the
convenience.

\begin{thm}\label{T1}
Let $A$ be a unital separable simple \CA\, with $TR(A)\le 1.$
Suppose that $u\in U_0(A).$ Then the following are equivalent:

{\rm (1)} $u\in CU(A);$

{\rm (2)} $\Delta(u)=0;$

{\rm (3)} for some piecewise continuous path of unitaries $\{u(t): t\in
[0,1]\}\subset A$ with $u(0)=u$ and $u(1)=1_A,$
$$
R(\{u(t)\})\in \overline{\rho_A(K_0(A))},
$$

{\rm (4)} for any piecewise continuous path of unitaries $\{u(t): t\in
[0,1]\}\subset A$ with $u(0)=u$ and $u(1)=1_A,$
$$
R(\{u(t)\})\in \overline{\rho_A(K_0(A))}.
$$

{\rm (5)} there are  $h_1,h_2,...,h_m\in A_{s.a.}$ such that
$$
u=\prod_{j=1}^m \exp(i h_j)\andeqn \sum_{j=1}^m \widehat{h_j}\in
\overline{\rho_A(K_0(A))}.
$$

{\rm (6)} $\sum_{j=1}^m \widehat{h_j}\in \overline{\rho_A(K_0(A))}$
for any $h_1,h_2,..., h_m\in A_{s.a.}$ for which
$$
u=\prod_{j=1}^m exp(i h_j).
$$

\end{thm}

\begin{proof}
Equivalence of (2), (3), (4), (5) and (6) follows from the
definition of the determinant and follows from the Bott periodicy
(see \cite{HS}). The equivalence of (1) and (2) follows from 3.1 of
\cite{Th}.
\end{proof}

The following is a consequence of \ref{T1}.

\begin{thm}\label{TDert}
Let $A$ be a unital simple separable \CA\, with $TR(A)\le 1.$ Then
${\rm ker}\Delta=CU(A).$ The de la Harpe and Skandalis determinant
gives an isomorphism:
$$
\overline{\Delta}: U_0(A)/CU(A)\to
\text{Aff}(T(A))/\overline{\rho_A(K_0(A))}.
$$
Moreover, one has the following short exact (splitting) sequence
$$
\begin{array}{ccccccc}
0 &\to \text{Aff}(T(A))/\overline{\rho_A(K_0(A))}
&\stackrel{{\overline{\Delta}}^{-1}}{\to} U(A)/CU(A)&\to K_1(A) &\to
0.
\end{array}
$$

\end{thm}
(Note that $U_0(A)/CU(A)$ is divisible in this case, by 6.6 of \cite{Lnctr1}.)

\section{Exponentials and approximate unitary equivalence orbit of unitaries}

\begin{thm}\label{EX2-1}
Let $A$ be a unital simple \CA\, with $TR(A)\le 1$ and let $\gamma: C(\T)_{s.a} \to \text{Aff}(T(A))$ be a
(positive) affine continuous map.

For any $\ep>0,$ there exists $\dt>0$ and there exists a finite subset ${\cal F}\subset C(\T)_{s,a}$  satisfying the following:
If $v\in U_0(A)$ with
\beq\label{Ex2-1}
|\tau(f(u))-\gamma(f)(\tau)|&<&\dt,\,\,\,\rforal f\in {\cal F}\tand \tau\in T(A),\,\tand\\
{\rm dist}({\bar u}, {\bar v})&<&\dt\,\,\,{\rm in}\,\,\,
U_0(A)/CU(A).
\eneq
Then there exists a unitary $W\in U(A)$ such that
\beq\label{Ex2-2}
\|u-W^*vW\|<\ep.
\eneq
\end{thm}

\begin{proof}
The lemma follows immediately from 3.11 of \cite{LM}. See also 11.5 of \cite{Lnapp1} and
3.15 of \cite{LM}. Note that, in 3.15 of \cite{LM}, we can replace the given map $h_1$ (in this case a given unitary) by
a given map $\gamma.$
\end{proof}

\begin{cor}\label{Ex2}

Let $A$ be a unital simple \CA\, with $TR(A)\le 1$ and let $u\in U_0(A)$ be a unitary.
For any $\ep>0,$ there exists $\dt>0$ and there exists an integer $N\ge 1$  satisfying the following:
If $v\in U_0(A)$ with
\beq\label{ex2-1}
|\tau(u^k)-\tau(v^k)|&<&\dt,\,\,\,k=1,2,...,N\rforal \tau\in T(A)\tand\\
{\rm dist}({\bar u}, {\bar v})&<&\dt\,\,\,{\rm in}\,\,\,
U_0(A)/CU(A).
\eneq
Then there exists a unitary $W\in U(A)$ such that
\beq\label{ex2-2}
\|u-W^*vW\|<\ep.
\eneq

\end{cor}

\begin{proof}
Note that (\ref{ex2-1}),
\beq\label{ex2-3}
|\tau(u^{k})-\tau(v^{k})|<\dt\,\,\,k=\pm 1,\pm 2,...,\pm N.
\eneq
For any subset ${\cal G}\subset C(S^1)$ and any $\eta>0,$ there exists $N\ge 1$ and $\dt>0$ such that
$$
|\tau(g(u))-\tau(g(v))|<\eta\tforal \tau\in T(A)
$$
if (\ref{ex2-3}) holds.

Then the lemma follows from \ref{EX2-1} (or 3.16 of \cite{LM}).

\end{proof}

\begin{thm}\label{Ex3}
Let $A$ be a unital simple \CA\, with $TR(A)\le 1.$ Suppose that $u\in U_0(A),$ then, for any $\ep>0,$ there exists a selfadjoint element $a\in A_{s.a}$ such that
\beq\label{ex3-1}
\|u-\exp(ia)\|<\ep.
\eneq
\end{thm}

\begin{proof}
Since $u\in U_0(A),$ we may write
\beq\label{ex3-2-}
u=\prod_{j=1}^k \exp(i h_j).
\eneq
Let $M=\max\{\|h_j\|: j=1,2,...,k\}+1.$ Let $\dt>0$ and $N$ be
given in \ref{Ex2} for $u.$ We may assume that $\dt<1$ and $N\ge
3.$ We may also assume that $\dt<\ep.$ Since $TR(A)\le 1,$ there
exists a projection $p\in A$ and a \SCA\, $B\in A$ with $1_B=p$
such that $B\cong \oplus_{i=1}^m C(X_i, M_{r(i)}),$ where
$X_i=[0,1]$ or a point, and

\beq\label{ex3-2}
&&\|pu-up\|<{\dt\over{16NMk}}, \\\label{ex3-2+1}
&&\|(1-p)u(1-p)-(1-p)\prod_{j=1}^k\exp(i((1-p)h_j(1-p))\|<{\dt\over{16NMk}}, \\\label{ex3-2+2}
&& pup\in_{{\dt\over{16NMk}}} B\andeqn
\tau(1-p)<{\dt\over{2NMk}}\tforal \tau\in T(A).
\eneq

There exist unitary $u_1\in B$ such that
\beq\label{ex3-3}
\|pup-u_1\|<{\dt\over{8NMk}}
\eneq
Put  $u_2=(1-p)\prod_{j=1}^k\exp(i(1-p)h_j(1-p)).$ Since  $u_1\in
B,$ it is well known that there exists a selfadjoint element $b\in
B_{s.a}$ such that
\beq\label{ex3-4}
\|u_1-p\exp(ib)\|<{\dt\over{16NMk}}.
\eneq
Let $v_0=(1-p)+p\exp(ib)$ and $u_0=p\exp(ib)+u_2.$  Then, by
(\ref{ex3-2}), (\ref{ex3-2+1}), (\ref{ex3-3}) and (\ref{ex3-4}),
\beq\label{ex3-5-1}
\|u_0-u\|&<&\|u-pup-(1-p)u(1-p)\|\\
&&+\|(pup-p\exp(ib))+((1-p)u(1-p)-u_2)\|\\
\label{ex3-5+2}
&<&{\dt\over{8NMk}}+{\dt\over{8NMk}}+
{\dt\over{16NMk}}={3\dt\over{8NMk}}.
\eneq
and
\beq\label{ex3-5}
u_0v_0^*=\prod_{j=1}^k \exp(i(1-p)h_j(1-p)).
\eneq
Note that
\beq\label{ex3-6}
|\tau(\sum_{j=1}^k (1-p)h_j(1-p))|&\le &\sum_{j=1}^k |\tau((1-p)h_j(1-p))|\\
&=&k\tau(1-p)\max\{\|h_j\|: j=1,2,...,k\}<\dt/16N
\eneq
for all $\tau\in T(A).$
It follows that
\beq\label{ex3-7}
{\rm dist}({\bar u_0}, {\bar v_0})<\dt/16N\,\,\,{\rm in}\,\,\, U_0(A)/CU(A).
\eneq
It follows from that
\beq\label{ex3-8}
{\rm dist}({\bar u}, {\bar v_0})<\dt/8N.
\eneq
On the other hand, for each $s=1,2,...,N,$ by  (\ref{ex3-5}), (\ref{ex3-5+2}) and (\ref{ex3-2+2})
\beq\label{ex3-9}
|\tau(u^s)-\tau(v_0^s)|&\le &
|\tau(u^s)-\tau(u_0^s)|+|\tau(u_0^s)-\tau(v_0^s)|\\\label{ex3-9+}
&\le & \|u^s-u_0^s\|+|\tau((1-p)-(1-p)\prod_{j=1}^k \exp(i(1-p)sh_j(1-p)))|\\
& \le & N\|u-u_0\|+2  \tau(1-p)\\
&<& {3\dt\over{8Mk}}+{\dt\over{MNk}}<\dt
\eneq
for all $\tau\in T(A).$ From the above inequality and  (\ref{ex3-8})
and applying \ref{Ex2}, one obtains a unitary $W\in U(A)$ such
that
\beq\label{ex3-10}
\|u-W^*v_0W\|<\ep.
\eneq
Put $a=W^*((1-p)+b)W.$  Then
\beq\label{ex3-11}
\|u-\exp(ia)\|<\ep.
\eneq

\end{proof}

Note that Theorem \ref{Ex3} does not assume that $A$ is amenable, in
particular, it may not be a simple AH-algebra. The proof used a kind
of uniqueness theorem for unitaries in a unital simple \CA\, $A$
with $TR(A)\le 1.$ This bring us to the following theorem which is
an immediate consequence of \ref{Ex2}.

\begin{thm}\label{Uni}
Let $A$ be a unital simple \CA\, with $TR(A)\le 1.$
Let $u$ and $v$ be two unitaries in $U_0(A).$ Then they are approximately
unitarily equivalent if and only if
\beq\label{Uni1}
\Delta(u)&=&\Delta(v)\andeqn\\\label{Uni2}
\tau(u^k)&=&\tau(v^k)\,\,\,\tforal \tau\in T(A),
\eneq
$k=1,2,....$
\end{thm}

Since ${\Delta}: U_0(A)/CU(A)\to
Aff(T(A))/\overline{\rho_A(K_0(A))}$ is an isomorphism, one may
ask if (\ref{Uni2}) implies that $\Delta(u)=\Delta(v)?$ In other
words, would  $\tau(f(u))=\tau(f(u))$ for all $f\in C(S^1)$
imply that $\Delta(u)=\Delta(v)?$  This becomes a question only in
the case that $\overline{\rho_A(K_0(A))}\not=Aff(T(A)).$ Thus we
would like to recall the following:

\begin{thm}\label{Dense}{\rm (cf. Theorem \cite{Lnplms})}\\ Let $A$ be a
unital simple \CA\, with $TR(A)\le 1.$ Then the following are
equivalent:

{\rm (1)}  $TR(A)=0,$

{\rm (2)}  $\overline{\rho_A(K_0(A))}=Aff(T(A))$ and

{\rm (3)} $CU(A)=U_0(A).$

\end{thm}

However, when $TR(A)=1,$ at least, one has the following:

\begin{prop}\label{div}
Let $A$ be a unital simple infinite dimensional \CA\, with $TR(A)\le 1.$
If $a\in {\overline{\rho_A(K_0(A))}},$ then
\beq\label{div1}
ra\in {\overline{\rho_A(K_0(A))}}
\eneq
for all $r\in \R.$ In fact, ${\overline{\rho_A(K_0(A))}}$ is a closed $\R$-linear subspace of
$Aff(T(A)).$

\end{prop}

\begin{proof}
Note that ${\overline{\rho_A(K_0(A))}}$ is an additive subgroup of $Aff(T(A)).$
It suffices to prove the following: Given any projection $p\in A,$  any real number $0<r_1<1$ and $\ep>0,$
there exists a projection $q\in A$ such that
\beq\label{dvi2}
|r_1\tau(p)-\tau(q)|<\ep\tforal \tau\in T(A).
\eneq
Choose $n\ge 1$ such that
\beq\label{dvi3}
|m/n-r_1|<\ep/2\andeqn 1/n<\ep/2
\eneq
for some $1\le m<n.$

Note that $TR(pAp)\le 1.$ By Theorem 5.4 or  Lemma 5.5 of \cite{Lnctr1}, there are mutually orthogonal projections
$q_0, p_1,p_2,...,p_n$ with $[q_0]\le [p_1]$ and $[p_1]=[p_i],$ $i=1,2,...,n$ and $\sum_{i=1}^n p_i+q_0=p.$
Put $q=\sum_{i=1}^m p_i.$ We then compute that
\beq\label{dvi4}
|r_1\tau(p)-\tau(q)|<\ep\tforal \tau\in T(A).
\eneq

\end{proof}

\begin{thm}\label{CEX}
Let $A$ be a unital simple infinite dimensional \CA\, with $TR(A)= 1.$
Then there exist  unitaries $u, v\in U_0(A)$ with
$$
\tau(u^k)=\tau(v^k)\tforal \tau\in T(A),\,\,\,k=0,\pm 1, \pm 2,...,\pm n,...
$$
such that $ \Delta(u)\not=\Delta(v).$
In particular, $u$ and $v$ are not approximately unitarily equivalent.
\end{thm}

\begin{proof}
Since we assume that $TR(A)=1,$ then, by \ref{Dense},
$Aff(T(A))\not=\overline{\rho_A(K_0(A))}$ and $U_0(A)/CU(A)$ are 
not trivial.

Let $\kappa_1,\, \kappa_2: K_1(C(\T))\to U_0(A)/CU(A)$ be two different \hm s.
Fix   an affine continuous map $s:
T(A)\to T_{f}(C(\T)),$ where $T_{f}(C(\T))$ is the space of
strictly positive normalized Borel measures on $\T.$
Denote by $\gamma_0: \text{Aff}(T(C(\T)))\to \text{Aff}(T(A))$ the positive affine continuous map
induced by $\gamma_0(f)(\tau)=f(s(\tau))$ for all $f\in  \text{Aff}(T(C(\T)))$ and $\tau\in T(A).$
Let
$$
\overline{\gamma_0}: U_0(C(\T))/CU(C(\T))=\text{Aff}(T(C(\T)))/\Z\to
\text{Aff}(T(A))/\overline{\rho_A(K_0(A))}=U_0(A)/CU(A)
$$
be the map induced by $\gamma_0.$
Write
$$
U(C(\T))/CU(C(\T))=U_0(C(\T))/CU(C(\T))\oplus K_1(C(\T)).
$$
Define $\lambda_i: U(C(\T))/CU(C(\T))\to U_0(A)/CU(A)$ by
$$\lambda_i(x\oplus z)=\overline{\gamma_0}(x)+\kappa_i(z)$$ for $x\in U_0(C(\T))/CU(C(\T))$ and
$z\in K_1(C(\T)),$ $i=1,2.$
 It follows
from 8.4 of \cite{Lnncl} that there are two unital monomorphisms
$\phi_1, \phi_2: C(\T)\to A$ such that
\beq\label{CEX1}
(\phi_1)_{*i}=0,\, \phi_i^{\ddag}=\lambda_i\andeqn
\phi_i^{\natural}=s,
\eneq
$i=1,2.$ Let $z$ be the standard unitary generator of $C(S^1).$
Define $u=\phi_1(z)$ and $v=\phi_2(z).$ Then $u, \, v\in U_0(A).$
The condition that $\phi_i^{\natural}=s$ implies that
$$
\tau(u^k)=\tau(v^k)\tforal \tau\in T(A), \,\,\,k=0, \pm 1, \pm 2,..., \pm n,....
$$
But since $\lambda_1\not=\lambda_2,$
$$
\Delta(u)\not=\Delta(v).
$$
Therefore $u$ and $v$ are not approximately unitarily equivalent.

\end{proof}

\begin{rem}\label{R3}

{\rm Given any continuous affine map $s: T(A)\to T_{f}(C(\T)),$ let
$\gamma_0: \text{Aff}(T(C(\T)))\to \text{Aff}(T(A))$ by defined by
$\gamma_0(f)(\tau)=f(s(\tau))$ for all $f\in \text{Aff}(T(C(\T)))$
and $\tau\in T(A).$ This further induces a \hm\, $\lambda:
U_0(C(\T))/CU(C(\T))\to U_0(A)/CU(A).$

Given any element $x\in \text{Aff}(T(A))/\overline{\rho_A(K_0(A))},$
the proof of the above theorem actually says that there is a unitary
$u\in U_0(A)$ such that $\Delta(u)=x$ and
$$
\tau(f(u))=f(s(\tau))
$$
for all $f\in C(\T)_{s.a}$ and $\tau\in T(A).$ Moreover, $u$ induces
$\lambda.$

}

\end{rem}

\section{Approximated by unitaries with finite spectrum}

Now we consider the problem when a unitary $u\in U_0(A)$ in a unital simple infinite dimensional
\CA\, $A$ with $TR(A)\le 1$ can be approximated by unitaries with finite spectrum. When
$TR(A)=0,$ $A$ has real rank zero, it was proved (\cite{LnFU}) that every unitary in $U_0(A)$ can be
approximated by unitaries with finite spectrum. When, $TR(A)=1,$ even a selfadjoint element in $A$ may not be approximated by those selfadjoint with finite spectrum. As stated in \ref{Dense}, in this case,
$\rho_A(K_0(A))$ is not dense in $Aff(T(A)).$ It turns out that that is the only issue.

\begin{lem}\label{app0}
Let $A$ be a unital separable simple infinite dimensional \CA\, with $TR(A)\le 1$ and let
$h\in A$ be a self-adjoint element. Then $h$ can be approximated by
self-adjoint elements with finite spectrum if and only if
$\widehat{h^n}\in \overline{\rho_A(K_0(A))},$ $n=1,2,....$
\end{lem}

\begin{proof}
If $h$ can be approximated by self-adjoint elements so can  $h^n.$ By \ref{div},
$\overline{\rho_A(K_0(A))}$ is a closed linear subspace.
Therefore $\widehat{h^n}\in \overline{\rho_A(K_0(A))}$ for all $n.$

Now we assume that $\widehat{h^n}\in \overline{\rho_A(K_0(A))},$
$n=1,2,....$ The Stone-Weierstrass theorem implies that
$\widehat{f(h)}\in \overline{\rho_A(K_0(A))}$ for all real-value
functions $f\in C(sp(h)).$ For any $\ep>0,$ by Lemma 2.4 of
\cite{Lnctr1}, there is $f\in C(sp(x))_{s.a.}$ such that
$$
\|f(h)-h\|<\ep
$$
and $sp(f(h))$ consists of a union of finitely many closed
intervals and finitely many points.

Thus, to simplify notation, we may assume that $X=sp(h)$ is a
union of finitely many intervals and finitely many points. Let
$\psi: C(X)\to A$ be the \hm\, defined by $\psi(f)=f(h).$ Let $s:
T(A)\to T_{f}(C(X))$ be the affine map defined by
$f(s(\tau))=\psi(f)(\tau)$ for all $f\in \text{Aff}(C(X))$ and
$\tau\in T(A).$

 Let $B$
be a unital simple AH-algebra with real rank zero, stable rank one
and
$$
(K_0(B), K_0(B)_+, [1_B], K_1(B))\cong (K_0(A), K_0(A)_+, [1_A],
K_1(A)).
$$
In particular, $K_0(B)$ is weakly unperforated. The proof of Theorem
10.4 of \cite{Lnctr1} provides a unital \hm\, $\imath: B\to A$ which
carries the above identification. This can be done by applying
Proposition 9.10 of \cite{Lnctr1} and the uniqueness theorem Theorem
8.6 of \cite{Lnctr1}, or better by corollary 11.7 of \cite{Lnapp1}
because $TR(B)=0,$ the maps $\phi^{\ddag}$ is not needed since
$U(B)=CU(B)$
%(see the first part of the proof of \ref{app})
and the
map on traces is determined by the map on $K_0(B).$ This also
follows immediately from Lemma 8.5 of \cite{Lnncl}.

Note that $\text{Aff}(T(B))=\overline{\rho_B(K_0(B))}.$ By
identifying $B$ with a unital \SCA\, of $A,$ we may write
$\overline{\rho_B(K_0(B))}=\overline{\rho_A(K_0(A))}.$

Let $\psi^{\natural}: \text{Aff}(T(C(X)))\to
\overline{\rho_A(K_0(A))}$ be the map induced by $\psi.$ This gives
a an affine map $\gamma: \text{Aff}(T(C(X)))\to
\overline{\rho_B(K_0(B))}.$  It follows from Lemma 5.1 of
\cite{LnMZ} that there exists a unital monomorphism $\phi: C(X)\to
B$ such that
$$
\imath\circ \phi_{*0}=\psi_{*0}\andeqn (\imath\circ
\phi)^{\natural}=\psi^{\natural},
$$
where $(\imath\circ \phi)^{\natural}: \text{Aff}(T(C(X)))\to
\text{Aff}(T(A))$ defined by $(\imath\circ
\phi)^{\natural}(a)(\tau)=\tau(\imath\circ \phi)(a)$ for all $a\in
A_{s.a.}.$  It follows from Corollary 11.7 of \cite{Lnapp1} that
$\psi$ and $\imath\circ \phi$ are approximately unitarily
equivalent. On the hand, since $B$ has real rank zero, $\phi$ can
be approximated by \hm s with finite dimensional range. It follows
that $h$ can be approximated by self-adjoint elements with finite
spectrum.

\end{proof}

\begin{thm}\label{app}
Let $A$ be a unital separable simple infinite dimensional \CA\, with $TR(A)\le 1$ and let
$u\in U_0(A).$ Then $u$ can be approximated by unitaries with finite
spectrum if and only if $u\in CU(A)$ and
$$
\widehat{u^n+ (u^{n})^*}, \,i(\widehat{u^n-(u^{n})^*})\in
\overline{\rho_A(K_0(A))},\,\,\,n=1,2,....
$$
\end{thm}

\begin{proof}

Suppose that there exists a sequence of unitaries $\{u_n\}\subset A$
with finite spectrum such that
$$
\lim_{n\to\infty}u_n=u.
$$
There are mutually orthogonal projections
$p_{1,n},p_{2,n},...,p_{m(n),n}\in A$ and complex numbers
$\lambda_{1,n}, \lambda_{2,n},...,\lambda_{m(n),n}\in \C$ with
$|\lambda_{i,n}|=1,$ $i=1,2,...,m(n,)$ and $n=1,2,...,$ such that
$$
\lim_{n\to\infty}\|u-\sum_{i=1}^{m(n)} \lambda_{i,n}p_{i,n}\|=0.
$$
It follows that
$$
\lim_{n\to\infty}\|((u^*)^n+u^n)-\sum_{i=1}^{m(n)}2{\rm
Re}(\lambda_{i,n})p_{i,n}\|=0.
$$
By \ref{div},
$$
 \sum_{i=1}^{m(n)}2{\rm
Re}(\lambda_{i,n})\widehat{p_{i,n}}\in \overline{\rho_A(K_0(A))}.
$$
Thus  $\widehat{Re(u^n})\in \overline{\rho_A(K_0(A))}.$ Similarly, $\widehat{Im(u^n)}\in \overline{\rho_A(K_0(A))}.$

 To
show that $u\in CU(A),$ consider a unitary $v=\sum_{i=1}^m \lambda_i
p_n,$ where $\{p_1,p_2,...,p_m\}$ is a set of mutually orthogonal
projections such that $\sum_{i=1}^m p_j=1,$ and where
$|\lambda_i|=1,$ $i=1,2,...,m.$  Write $\lambda_j=e^{i\theta_j}$ for
some real number $\theta_j,$ $j=1,2,....$ Define
$$
h=\sum_{j=1}^m \theta_jp_j.
$$
Then
$$
v=\exp(i h).
$$
By \ref{div},  $\widehat{h}\in \overline{\rho_A(K_0(A))}.$ It follows
from \ref{T1} that $v\in CU(A).$ Since $u$ is a limit of those
unitaries with finite spectrum, $u\in CU(A).$

Now assume $u\in CU(A)$ and $\widehat{u^n+(u^{n})^*},\,
i(\widehat{u^n-(u^{n})^*})\in \overline{\rho_A(K_0(A))}$ for
$n=1,2,....$ If $sp(u)\not=\T,$ then the problem is reduced to the
case in \ref{app0}. So we now assume that $sp(u)=\T.$ Define  a
unital monomorphism $\phi: C(\T)\to A$ by $\phi(f)=f(u).$ By the
Stone-Weirestrass theorem and \ref{div}, every real valued funtion
$f\in C(\T),$ $\widehat{\phi(f)}\in \overline{\rho_A(K_0(A))}.$

As in the proof of \ref{app0}, one obtains a unital \SCA\,
$B\subset A$  which is a unital simple AH-algebra with tracial
rank zero such that the embedding $\imath: B\to A$ gives an
identification:
$$
(K_0(B), K_0(B)_+, [1_B], K_1(B))=(K_0(A), K_0(A)_+, [1_A], K_1(A)).
$$
Moreover, by Lemma 5.1 of \cite{LnMZ} that there is a unital
monomorphism $\psi: C(\T)\to B$ such that
$$
\psi_{*1}=0\andeqn (\imath\circ \psi)^{\natural}=\phi^{\natural}.
$$
Note also
$$
(\imath\circ \psi)^{\ddag}=\phi^{\ddag}
$$
(both are trivial, since $u\in CU(A)$).

 It follows from \ref{Uni} (see also Theorem 11.7 of \cite{Lnapp1}) that
$\imath\circ \psi$ and $\phi$ are approximately unitarily
equivalent. However, since $\psi_{*1}=0,$ in $B,$  by \cite{LnFU},
$\psi$ can be approximated by \hm s with finite dimensional range.
It follows that $u$ can be approximated by unitaries with finite
spectrum.

\end{proof}

If $A$ is a finite dimensional simple  \CA, then $TR(A)=0.$ Of course, every unitary
in $A$ has finite spectrum.  But $CU(A)\not=U_0(A).$ To unify the two cases, we note that $K_0(A)=\Z.$
Instead of using $\overline{\rho_A(K_0(A))},$ one may consider  the following definition:

\begin{df}\label{DD}
{\rm Let $A$ be a unital \CA. Denote by $V(\rho_A(K_0(A))),$ the
closed $\R$-linear subspace of $\text{Aff}(T(A))$ generated by
$\rho_A(K_0(A)).$ Let $\Pi: \text{Aff}(T(A))\to
\text{Aff}(T(A))/V(\rho_A(K_0(A)))$ be the quotient map. Define
the new determinant
$$
{\tilde \Delta}: U_0(A)\to \text{Aff}(T(A))/V(\rho_A(K_0(A)))
$$
by
$$
{\tilde \Delta}(u)=\Pi\circ \Delta(u)\tforal u\in U_0(A).
$$

Note that if $A$ is a finite dimensional \CA\,
$\text{Aff}(T(A))=V(\rho_A(K_0(A))).$ Thus ${\tilde \Delta}=0.$ If
$A$ is a unital simple infinite dimensional \CA\, with $TR(A)\le
1,$ by \ref{div},
$$
V(\rho_A(K_0(A)))=\overline{\rho_A(K_0(A))}.
$$

}

\end{df}

\begin{df}

{\rm Suppose that $u\in A$ is a unitary with $X=\text{sp}(u).$
Then it induces a positive affine continuous map from
$\gamma_0:C(X)_{s.a.}\to \text{Aff}(T(A))$ defined by
$$
\gamma_0(f(u))(\tau)=\tau(f(u))
$$
for all $f\in C(X)_{s.a.}$ and all $\tau\in T(A).$ Let $\Pi:
\text{Aff}(T(A))\to \text{Aff}(T(A))/V(\rho_A(K_0(A))).$ Put
$\Gamma(u)=\Pi\circ \gamma_0.$ Then $\Gamma(u)$ is a map from
$C(X)_{s.a.}$ into $\text{Aff}(T(A))/V(\rho_A(K_0(A))).$

It is clear that, $\Gamma(u)=0$ if and only if
$\widehat{u^n+(u^n)^*},\, i(\widehat{u^n+(u^n)^*})\in
V(\rho_A(K_0(A)))$ for all $n\ge 1.$

 }

\end{df}

Thus, we may state the following:

\begin{cor}\label{C1}
Let $A$ be a unital simple \CA\, with $TR(A)\le 1$ and let $u\in U_0(A).$ Then
$u$ can be approximated by unitaries with finite spectrum if and only if
$$
{\tilde \Delta}(u)=0\andeqn \Gamma(u)=0.
$$
\end{cor}

\begin{NN}
{\rm  Suppose that $u=\exp(i h)$ for some self-adjoint element $h\in
A.$ If $u\in CU(A),$ then, by \ref{T1}, ${\tilde \Delta}(u)=0,$
i.e., $\widehat{h}\in V(\rho_A(K_0(A))).$  So one may ask if there
are unitaries with ${\tilde\Delta}(u)=0$  but $\Gamma(u)\not=0.$
Proposition \ref{No} below says that this could happen.

}
\end{NN}

\begin{prop}\label{No}
For any unital  separable simple  \CA\, $A$ with $TR(A)=1,$ there
is a unitary $u$ with ${\tilde \Delta}(u)=0$ {\rm (}or $u\in
CU(A)${\rm )} such that $\Gamma(u)\not=0$ and which is not a limit
of unitaries with finite spectrum.

\end{prop}

\begin{proof}
Let $e\in A$ be a non-zero projection such that there is a
projection $e_1\in (1-e)A(1-e)$ such that $[e]=[e_1].$ Then
$TR(eAe)\le 1$ by 5.3 of \cite{Lnplms}. Since $A$ does not have real rank zero, one has
$TR(eAe)=1.$

It follows from \ref{Dense} that
$$
\text{Aff}(T(eAe))\not=\overline{\rho_A(K_0(eAe))}=\overline{\rho_A(K_0(A))}.
$$
Choose $h\in (eAe)_{s.a.}$ with $\|h\|\le 1$ such that $h$ is not a
norm limit of self-adjoint elements with finite spectrum.

If $\widehat{h}\in \overline{\rho_A(K_0(eAe))},$ then define
$$
u=\exp(i h).
$$
Then, $\Delta(u)=0$ and by Theorem \ref{T1}, $u\in CU(A).$ Since $h$ can not be
approximated by self-adjoint elements with finite spectrum, nor $u$
can be approximated by unitaries with finite spectrum since
$h=(1/i)\log (u)$ for a continuous branch of the logarithm (note that $sp(u)\not=\T$).

Now suppose that $\widehat{h}\not\in \overline{\rho_A(K_0(eAe))}.$

We also have, by \ref{div}, $2\pi \widehat{h}\not\in \overline{\rho_A(K_0(A))}.$  We claim that there is
a rational number $0<r\le 1$ such that
$r\widehat{h^2}-2\pi\widehat{h}\not\in \overline{\rho_A(K_0(eAe))}.$

In fact, if $\widehat{h^2}\in \overline{\rho_A(K_0(eAe))},$ then the claim follows easily.  So we assume
that $\widehat{h^2}\not\in \overline{\rho_A(K_0(eAe))}.$
Suppose that, for some $0<r_1<1,$  $r_1\widehat{h^2}-2\pi\widehat{h}\in \overline{\rho_A(K_0(eAe))}.$
Then  $(1-r_1)\widehat{h^2}\not\in \overline{\rho_A(K_0(eAe))}.$  Hence
$$
\widehat{h^2}-2\pi \widehat{h}=(1-r_1)\widehat{h^2}+(r_1\widehat{h^2}-2\pi\widehat{h})\not\in \overline{\rho_A(K_0(eAe))}.
$$
This proves the claim.

Now define $h_1=rh+ 2\pi e_1-w^*rhw,$ where $w\in A$ is a unitary such
that $w^*ew=e_1.$ Put
$$
u=\exp(ih_1)
$$
%Then $\tau(h_1)=2\pi\tau(e_1).$ By Theorem 5.4 of \cite{Lnctr1}, $A$
%is tracially approximately divisible.
It follows from \ref{div} that
$$
2\pi \widehat{e_1}\in \overline{\rho_A(K_0(eAe))}.
$$
Thus $\tau(h_1)=2\pi\tau(e_1)\in \overline{\rho_A(K_0(eAe))}.$
Therefore, by \ref{T1}, $u\in CU(A).$ Since
\beq
\widehat{h_1^2}&=&\widehat{r^2h^2}+4\pi^2\, \widehat{e_1}-4\pi
\widehat{rh}+\widehat{r^2h^2}\\
&=& 2r(r\widehat{h^2}-2\pi\widehat{h})-4\pi^2 \widehat{e_1}\not\in
\overline{\rho_A(K_(A))}.
\eneq
%($\text{Af}(T(A))/\overline{\rho_A(K_0(A))}$ is torsion free).
Therefore, by \ref{app0},  $h_1$ can not be approximated by self-adjoint elements
with finite spectrum. It follows that $u$ can not be approxiamted by
unitaries with finite spectrum.

\end{proof}

Another question is whether  $\Gamma(u)=0$  is sufficient for
$\Delta(u)=0.$   For the case that ${\rm sp}(u)\not=\T,$ one has the
following. But in general, \ref{LP} gives a negative answer.

\begin{prop}
Let $A$ be a unital separable simple \CA\, with $TR(A)\le 1.$
Suppose that $u\in U_0(A)$ with ${\rm sp}(u)\not=\T.$ If
$\Gamma(u)=0,$ then ${\tilde\Delta}(u)=0,$ $u\in CU(A)$ and $u$ can be
approximated by unitaries with finite spectrum.
\end{prop}

\begin{proof}
Since ${\rm sp}(u)\not=\T,$ there is a real valued continuous
function $f\in C({\rm sp}(u))$ such that $u=\exp(i f(u)).$ Thus the
condition that $\Gamma(u)=0$ implies that $\widehat{f(u)}\in
\overline{\rho_A(K_0(A))}.$ By \ref{T1}, $u\in CU(A).$

\end{proof}

\begin{prop}\label{LP}
Let $A$ be a unital infinite dimensional separable simple \CA\, with
$TR(A)=1.$ Then there are unitaries $u\in U_0(A)$ with $\Gamma(u)=0$
such that $u\not\in CU(A).$ In particular, ${\tilde \Delta}(u)\not=0$ and $u$ can not be
approximated by unitaries with finite spectrum.

\end{prop}
\begin{proof}
There exists a unital \SCA\, $B\subset A$ with tracial rank zero
such that the embedding gives the following identification:
$$
(K_0(B), K_0(B)_+, [1_B], K_1(B))=(K_0(A), K_0(A)_+, [1_A],
K_1(A)).
$$
Note that $\text{Aff}(T(B))=\overline{\rho_B(K_0(B))}=
\overline{\rho_A(K_0(A))}.$

Let $w\in U_0(B)$ be a unitary with ${\rm sp}(w)=\T.$ Thus
$\Gamma(w)=0.$ Let $\gamma: \text{Aff}(T(C(\T)))\to
\text{Aff}(T(A))$ be given by $\gamma(f)(\tau)=\tau(f(u))$ for
$f\in C(T)_{s.a.}$ and $\tau\in T(A).$ Since $TR(A)=1,$ by
\ref{T1}, there are unitaries $u_0\in U_0(A)\setminus CU(A).$ By
the proof of \ref{CEX} (see also \ref{R3}), there is a unitary
$u\in U_0(A)$ such that
$$
\overline{u}=\overline{u_0}\andeqn
$$
$$
\tau(f(u))=\tau(f(w))\tforal \tau\in T(A)
$$
and for all $f\in C(\T)_{s.a.}.$ Thus $\tilde{\Delta}(u)\not=0$
and $\Gamma(u)=\Gamma(w)=0.$ By \ref{app}, $u$ can not be
approximated by unitaries with finite spectrum.

\end{proof}

\bibliographystyle{amsalpha}
\bibliography{}

\end{document}